\documentclass[12pt]{amsart}

\usepackage{amssymb}
\usepackage{amsmath}
\usepackage{amsfonts}
\usepackage{graphicx}
\usepackage{amsthm}
\usepackage{enumerate}
\usepackage[mathscr]{eucal}
\usepackage{mathrsfs}
\usepackage{verbatim}
\usepackage{yhmath}
\usepackage{epstopdf}

\makeatletter
\@namedef{subjclassname@2010}{%
  \textup{2010} Mathematics Subject Classification}
\makeatother

\numberwithin{equation}{section}
\numberwithin{figure}{section}

\theoremstyle{plain}
\newtheorem{theorem}{Theorem}[section]
\newtheorem{lemma}[theorem]{Lemma}
\newtheorem{proposition}[theorem]{Proposition}

\theoremstyle{plain}

\numberwithin{equation}{section}

\theoremstyle{remark}
\newtheorem{remark}[theorem]{Remark}

\DeclareMathOperator{\supp}{supp}
\DeclareMathOperator{\dist}{dist}

\begin{document}
\date{}

\title
[Improved Weyl remainder for the planar disk]{An improved remainder estimate in the Weyl formula for the planar disk}
 
\author{Jingwei Guo}
\address{School of Mathematical Sciences\\
University of Science and Technology of China\\
Hefei, 230026\\ P.R. China\\} \email{jwguo@ustc.edu.cn}
\author{Weiwei Wang}
\address{School of Mathematical Sciences\\
University of Science and Technology of China\\
Hefei, 230026\\ P.R. China\\} \email{wawnwg123@163.com}
\author{Zuoqin Wang}
\address{School of Mathematical Sciences\\
University of Science and Technology of China\\
Hefei, 230026\\ P.R. China\\} \email{wangzuoq@ustc.edu.cn}

\thanks{The authors are partially supported by the NSFC Grant No. 11571131. J.G. is also partially supported by the NSFC Grant No. 11501535. 
Z.W. is also partially supported by the NSFC Grant No. 11526212. 
}

\begin{abstract}
In \cite{colin}, Y. Colin de Verdi\`{e}re proved that the remainder term in the two-term Weyl formula for the eigenvalue counting function for the Dirichlet Laplacian associated with the planar disk is of order $O(\lambda^{2/3})$. In this paper, by combining with the method of exponential sum estimation, we will give a sharper remainder term estimate $O(\lambda^{2/3-1/495})$.
\end{abstract}

\subjclass[2010]{Primary 35P20, 11P21}

\keywords{Laplace eigenvalue, Weyl law, Lattice point problem, exponential sum.}

\maketitle

\section{Introduction}\label{introduction}


Let $D \subset \mathbb R^2$ be a bounded planar\footnote{Although most results are valid in higher dimensions, we will only present the 2-dimensional version here.}  domain with piecewise smooth boundary. Denote by $\mu^2_1 \le \mu^2_2 \le \cdots  $ the eigenvalues of the  Laplacian associated to $D$ with either Dirichlet or Neumann boundary conditions. To describe the asymptotic growth of $\mu_n$ as $n$ goes to infinity, one is led to study the eigenvalue counting function $\mathcal N_D(\mu)=\#\{n\ |\ \mu_n <\mu \}$. In his seminal work in 1911,  H. Weyl  (\cite{weyl11}) proved
\[\mathcal N_D(\mu)=\frac{\mathrm{Area}(D)}{4\pi}\mu^2+o(\mu^2).\]
To improve the remainder term had been one of the most attractive problems in spectral theory for decades. A huge number of papers were published on this problem. One of the ultimate goals was to prove the following two-term asymptotic formula for $\mathcal N_D(\mu)$,  which was conjectured by H. Weyl in 1913 (\cite{weyl13}),
\[\mathcal N_D(\mu)=\frac{\mathrm{Area}(D)}{4\pi}\mu^2 \mp \frac{\mathrm{Length}(\partial D)}{4\pi}\mu+o(\mu),\]
where ``$-$" is for the Dirichlet boundary condition and ``$+$" is for the Neumann boundary condition. The two-term asymptotic formula above was finally proven by V. Ivrii \cite{ivrii} and by R. Melrose \cite{melrose} under some ``non-periodic billiard trajectory" assumptions. For more backgrounds and history of Weyl's law, we will refer to the nice survey papers \cite{ANPS}, \cite{ivrii2} and the references therein.

Although it is still a conjecture that any piecewise smooth planar domain satisfies the ``non-periodic billiard trajectory" assumption, people has already confirmed that a number of nice regions satisfy that property, and thus the two-term Weyl's law hold. Such regions include all bounded convex domains with analytic boundary and all bounded domains with piecewise-smooth concave boundary. Now supoose $D$ is such a region. A natural further investigation would then be to find the smallest constant $\kappa$ such that $\mathcal R_D(\mu)=O(\mu^\kappa)$, where
\[\mathcal R_D(\mu) := \mathcal N_D(\mu)-\frac{\mathrm{Area}(D)}{4\pi}\mu^2 \pm \frac{\mathrm{Length}(\partial D)}{4\pi}
\mu  \]
is the remainder term in the two-term Weyl's law. This problem is subtle since $\mathcal R_D(\mu)$ encodes the dynamical information of the billiard flow associated with the region $D$.
On one hand, there is no universal $\kappa<1$  so that $\mathcal R_D(\mu)=O(\mu^\kappa)$ holds for all planar domains: for each $\kappa <1$ V. Lazutkin and D. Terman constructed in \cite{LT} convex planar domains $D$ with specific billiard dynamics so that
$\mathcal R_D(\mu) \ne O(\mu^\kappa)$ (though the two-term Weyl's law still holds). On the other hand, $\kappa$ can be much smaller than one for specific domains. For example, when $D$ is a square in the plane, the problem of estimating $\mathcal R_{square}(\mu)$ is equivalent to the famous Gauss circle problem. According to the current record (due to M. Huxley \cite{Huxley}) on the latter problem, one has $\mathcal R_{square}(\mu)=O(\mu^{131/208+\varepsilon})$.

Another very interesting and typical example is the planar disk. It is well-known that the Laplacian eigenvalues of the planar disk (with Dirichlet boundary condition) are precisely the zeroes of the Bessel functions. By studying the asymptotic growth of large zeroes, in  \cite{colin} Y. Colin de Verdi\`{e}re was able to convert the problem of estimating $\mathcal R_{disk}(\mu)$ into a lattice point problem of the following region $\Omega$ (see Figure \ref{figure1.1}) which is non-convex and contains two cusp points. The explicit definition of the region $\Omega$ is
\begin{equation}\label{DomainOmega}
\Omega=\{(t,s)\in \mathbb{R}^2 : -1\leq t\leq 1,\, \max(0,-t)\leq s\leq g(t)\},
\end{equation}
where
\begin{equation}\label{def-g}
g(t)=(\sqrt{1-t^2}-t\arccos t)/\pi, \quad t\in [-1, 1].
\end{equation}

\begin{figure} 
\includegraphics[width=0.7\textwidth]{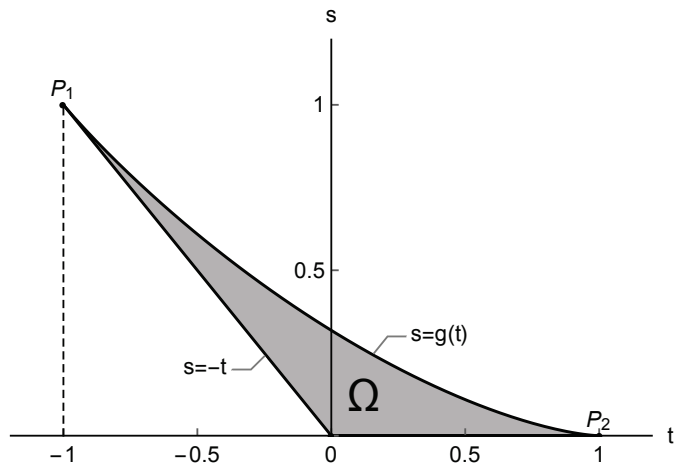}
\caption{The domain $\Omega$}\label{figure1.1}
\end{figure}

By studying the lattice point problem associated with the domain $\Omega$ above, Y. Colin de Verdi\`{e}re was able to prove\footnote{As mentioned in \cite{colin}, the same result was also announced in 1964 by N. Kuznecov and B. Fedosov in \cite{KF}.} $\kappa \le 2/3$, {\it i.e.}
\[\mathcal R_{disk}(\mu)=O(\mu^{2/3}).\]

In this paper, we will improve the power $2/3$ a bit further:
\begin{theorem}\label{spectrum-thm}
For the planar disk with either Dirichlet or Neumann boundary condition,
\[\mathcal R_{disk}(\mu)=O(\mu^{2/3-1/495}).\]
\end{theorem}
We remark that recently, S. Eswarathasan, I. Polterovich and J. Toth show in \cite{EPT} that for a class of planar regions that includes the planar ellipses, as $\mu \to \infty$ one has
\[
\frac 1{\mu}\int_{\mu}^{2\mu}\mathcal R_D(\tau)d\tau \gg \sqrt \mu.
\]
So in particular, one has $\kappa \ne 1/2$ for the planar disk.


Before we explain the strategy of the proof of the main theorem above, let us say some words on the classic lattice point problem.  Let $\Omega$ be a planar compact domain. One would like to count the number of lattice points (modulo a translation) in the dilated domain $\mu \Omega$ for large $\mu$, namely
\begin{equation*}
N_{\Omega}(\mu):=\#\left(\mu \Omega\cap \mathbb Z^2_{a,b}\right),
\end{equation*}
where $\mathbb Z^2_{a,b}=\mathbb Z^2+(a,b)$ is the  translation of $\mathbb Z^2$ by the vector $(a, b)$. Under very mild assumptions, it is easy to deduce that $N_\Omega(\mu) \sim \mathrm{Area}(\Omega)\mu^2$ and one is led to study the asymptotic behavior of the remainder term
\begin{equation}\label{remainder}
P_{\Omega}(\mu):= N_{\Omega}(\mu) - \mathrm{Area}(\Omega)\mu^2.
\end{equation}
This is known as \emph{the lattice point problem}.

A classic result is that $P_\Omega(\mu) = O(\mu^{2/3})$ under certain proper regularity and curvature conditions on $\partial\Omega$ if the origin is an interior point of a convex $\Omega$. This type of estimate with the power $2/3$ reaches back to a classic work \cite{vandercorput1920} of J.G. Van der Corput in 1920, in which domains having boundary curve of class $C^2$ and nowhere vanishing radius of curvature were considered. In fact, the same estimate holds for quite general domains after a ``generic'' rotation for almost all $(a, b)$. See for example Y. Colin de Verdi\`{e}re~\cite{colin77}. For more results of this type we refer interested readers to the first author~\cite{guo2, guo3} and the references therein.

Numerous works have been done in improving the power $2/3$ for various classes of domains during the last century. We will not give a detailed account of the history and all known results here, but just mention a few briefly. For a planar convex domain having sufficiently smooth boundary with nonzero curvature the best known bound, $O(\mu^{131/208+\varepsilon})$, is obtained by M.  Huxley in \cite{Huxley}. If the
boundary curve contains points with curvature zero and irrational slope, under certain assumptions about the Diophantine approximation of the slope by rationals, W. M{\"u}ller and W.~G. Nowak~\cite{mullernowak, mullernowak2} proved (among others) bounds of the form
\begin{equation*}
O(\max(\mu^{2/3-\gamma}, \mu^{7/11}(\log \mu)^{45/22})),
\end{equation*}
where $\gamma>0$ depends on the maximal order of vanishing of the curvature. As a consequence, they obtained the same bound after almost all rotations. Later the first author~\cite{guo2, guo3} proved bounds with absolute exponents less than $2/3$ for almost all rotations.

To get the the bound $O(\mu^{2/3})$, roughly speaking, a combination of the Poisson summation formula and (nowadays standard) oscillatory integral estimates would suffice. To get sharper bounds, however, one needs to additionally combine with methods of exponential sum estimation (like the classic Van der Corput's method or more sophisticated methods like the discrete Hardy--Littlewood method).

In order to prove Theorem \ref{spectrum-thm} we need to consider the lattice point problem of the domain $\Omega$ defined by \eqref{DomainOmega} (see Figure \ref{figure1.1}). This compact non-convex domain has two cusp points, {\it i.e.} $P_1(-1,1)$ and $P_2(1,0)$. The curvature of the curve $s=g(t)$ tends to infinity as one approaches the cusps. The slopes at $P_1$ and $P_2$ are rational. Due to these special features, no known results on the lattice point problem can be applied directly. In \cite{colin}, by using oscillatory integral methods it is proven that $P_\Omega=-\mu/2+O(\mu^{2/3})$ when $(a, b)=(0, -1/4)$.  In this paper we will modify the Van der Corput's method of exponential sums in literature (see for example W. M\"uller~\cite{mullerII}), by adding an extra parameter (in order to handle the unusual curvature), to prove the following sharper estimate.

\begin{theorem}\label{lattice-thm}
Suppose $(a, b)=(0, -1/4)$. For the lattice point remainder (defined by \eqref{remainder}) associated with the domain $\Omega$ (defined by \eqref{DomainOmega})  we have
\begin{equation}\label{LatticeRemainder}
P_{\Omega}(\mu)=-\mu/2+O(\mu^{2/3-1/495}).
\end{equation}
\end{theorem}
Once this theorem is proven, Theorem \ref{spectrum-thm} follows easily from a comparison argument between $\mathcal N_{disk}(\mu)$ and $N_\Omega(\mu)$  which is essentially due to Colin de Verdi\`{e}re.

\begin{remark}
Our goal in this paper is mainly to obtain bounds with an exponent strictly less than $2/3$ for both the eigenvalue counting problem and the lattice point counting problem. What we can get so far is the number $2/3-1/495$. It is very likely that one can further lower this number by iterating the estimation of exponential sums used in this paper. An interesting question would then be how far one can go by using the same method.
\end{remark}


{\it Notations:} We use the usual Euclidean norm $|x|$ for a point
$x\in \mathbb{R}^2$. $B(x, r)\subset \mathbb{R}^2$ represents the
Euclidean ball centered at $x$ with radius $r$. The norm of a matrix
$A\in \mathbb{R}^{2\times 2}$ is given by $\|A\|=\sup_{|x|=1}|Ax|$.
We set $e(s)=\exp(2\pi i s)$,
$\mathbb{Z}_{*}^{2}=\mathbb{Z}^{2}\setminus \{0\}$, and
$\mathbb{R}^2_*=\mathbb{R}^2\setminus \{0\}$. The Fourier transform
of $f\in L^1(\mathbb{R}^2)$ is $\widehat{f}(\xi)=\int f(x)
e\left(-\langle x, \xi\rangle\right) \, dx$.

We use the differential operators
\begin{equation*}
D^{\nu}_{x}=\frac{\partial^{|\nu|}}{\partial x_1^{\nu_1} \partial
x_2^{\nu_2}} \quad \left(\nu=(\nu_1, \nu_2)\in \mathbb{N}_0^2,
|\nu|=\nu_1+\nu_2 \right)
\end{equation*}
and the gradient operator $\nabla_x$. We often omit the subscript if
no ambiguity occurs.

For functions $f$ and $g$ with $g$ taking nonnegative real values,
$f\lesssim g$ means $|f|\leqslant Cg$ for some constant $C$. If $f$
is nonnegative, $f\gtrsim g$ means $g\lesssim f$. The Landau
notation $f=O(g)$ is equivalent to $f\lesssim g$. The notation
$f\asymp g$ means that $f\lesssim g$ and $g\lesssim f$.


The rest of this paper is organized as follows. In Section \ref{oscillatory-integral} we will prove some properties of the region $\Omega$ and provide both estimates and asymptotic formulas of certain integral along a part of $\partial\Omega$. In Section \ref{expsumsec} we will adjust the standard exponential sum estimation to our case which involves an extra parameter $K$ (representing the curvature in its later application). In Section \ref{nonvanishing-det} we justify the nonvanishing of certain determinants that is needed in applying the method of stationary phase.
Finally we will prove in Section \ref{lattice-point} the lattice point counting theorem, {\it i.e.} Theorem \ref{lattice-thm}, and in Section \ref{spec-to-lattice} the eigenvalue counting theorem, {\it i.e.} Theorem \ref{spectrum-thm}.


\section{Oscillatory Integral Results}\label{oscillatory-integral}

In this section we study the curve $\Gamma_g$ parametrized by $(t,
g(t))$ (with $g$ as defined in \eqref{def-g}) and provide both
estimates and asymptotic formulas of certain integral along
$\Gamma_g$ (see \eqref{key-integral} below).

The Gauss map, denoted by $\vec{n}$, maps each point $P\in \Gamma_g$
to a unit exterior normal $\vec{n}(P)\in S^1$. In particular if $P$
is not either endpoint then $\vec{n}(P)$ is in the cone
\begin{equation}
\mathfrak{C}_0:=\{(x_1, x_2)\in \mathbb{R}^2 : x_1, x_2>0,
1<x_2/x_1<\infty\}.\label{key-cone}
\end{equation}
For each $\xi\in \mathfrak{C}_0$ there exists a unique point
$x(\xi)=(t(\xi), g(t(\xi)))\in \Gamma_g$ where the normal is along
$\xi$. Denote by $K_{\xi}$ the curvature of the curve at that point,
whose value is nonzero, bounded from below, and very large as
$x(\xi)$ approaches either endpoint.

By Taylor's formula it is easy to prove that

\begin{lemma}\label{angle-curv} Let $P_1(-1,1)$ and $P_2(1,0)$ be endpoints of the curve
$\Gamma_g$. There exist two positive constants $C$ and $C'$ such
that if $P\in \Gamma_g$ is sufficiently close to $P_i$ ($i=1, 2$)
then
\begin{equation}
C/K_{\vec{n}(P)}\leq \mathfrak{A}_{\vec{n}(P), \vec{n}(P_i)} \leq
C'/K_{\vec{n}(P)}, \label{geometry2}
\end{equation}
where $\mathfrak{A}_{\vec{n}(P), \vec{n}(P_i)}$ denotes the angle in
$[0, \pi]$ between $\vec{n}(P)$ and $\vec{n}(P_i)$.
\end{lemma}

Then it is readily to get that

\begin{lemma}\label{curv-in-ball} There exist positive constants $c_1$, $A$, and $A'$ such that for
any $\xi\in\mathfrak{C}_0$ we have
\begin{equation*}
A K_\xi \leq K_\eta \leq A' K_\xi    \qquad \textrm{whenever
$\eta\in \mathfrak{C}_0$ and $\mathfrak{A}_{\eta, \xi}\leq
c_1/K_\xi$}.
\end{equation*}
\end{lemma}

For $\mu >1$ we next study the following oscillatory integral
associated with $\Gamma_g$
\begin{equation}
I(\mu, \xi):=\int_{-1}^{1} e^{-i\mu (\xi_1 t+\xi_2
g(t))}\left(\xi_2-\xi_1 g'(t) \right)
\,\textrm{d}t.\label{key-integral}
\end{equation}

We first have an estimate of $I(\mu, \xi)$ which was essentially
obtained (though not explicitly stated) in Colin de
Verdi\`{e}re~\cite{colin}'s Lemma 4.3. To prove it not much
essential change is needed of Colin de Verdi\`{e}re's original
proof. We refer interested readers to standard textbooks in harmonic
analysis for more similar results.

\begin{proposition}\label{estimate-key-integral}
There exist a positive constant $C$ and neighborhoods of points
$(1/\sqrt{2}, 1/\sqrt{2})$ and $(0, 1)$ such that if $\xi\in
S^1\cap\mathfrak{C}_0$ is in either neighborhood then
\begin{equation*}
\left|I(\mu, \xi)\right|\lesssim \left\{ \begin{array}{ll}
\mu^{-2/3}              & \textrm{if $\mu\leq C K_{\xi}^3$}\\
\mu^{-1/2}K_{\xi}^{-1/2} & \textrm{if $\mu\geq C K_{\xi}^3$}
\end{array}.\right.
\end{equation*}
\end{proposition}

Using this result one can obtain (as in \cite{colin}) the bound
$O(\mu^{2/3})$ of the remainder $P_{\Omega}(\mu)$ of the lattice
point problem.

We actually have asymptotics of $I(\mu, \xi)$, which will be needed
later when we prove a better bound.

\begin{proposition}\label{asym-R^2}
For any $\xi\in S^1\cap\mathfrak{C}_0$ we have
\begin{equation}
\begin{split}
I(\mu, \xi)&=(2\pi)^{1/2}e^{-\pi i/4} K_{\xi}^{-1/2}e^{-i\mu
\langle \xi, x(\xi)\rangle}\mu^{-1/2}+O\left(K_{\xi}^{5/2}\mu^{-3/2}\right)\\
&\quad\quad
+O\left(K_{\xi}\mu^{-1}\right)+O\left(\left(\left|\frac{\xi_2}{\xi_1}\right|+\left|\frac{\xi_2+\xi_1}{\xi_2-\xi_1}\right|\right)\mu^{-1}
\right).
\end{split}\label{asym-formula}
\end{equation}
\end{proposition}

\begin{proof}
Fix an arbitrary $\xi\in S^1\cap\mathfrak{C}_0$. Let $\chi_0$ be a
fixed smooth cut-off function whose value is $1$ on $B(0, 1)$ and
$0$ on the complement of $B(0, 2)$. By adding a partition of unity
$1\equiv\chi_1+\chi_2$ (to the integrand) where
\begin{equation*}
\chi_1(t, \xi)=\chi_0\left(\frac{|\vec{n}((t, g(t)))-\xi|}{c
K_{\xi}^{-1}} \right)  \textrm{ and } \chi_2(t, \xi)=1-\chi_1(t,
\xi)
\end{equation*}
with $c>0$ to be determined below, we get
\begin{equation*}
I(\mu, \xi)=I_1(\mu, \xi)+I_2(\mu, \xi),
\end{equation*}
where
\begin{equation*}
I_j(\mu, \xi):=e^{-i\mu \langle \xi, x(\xi)\rangle}\int_{-1}^{1}
e^{-i\mu h(t,\xi)}\chi_j(t, \xi)(\xi_2-\xi_1 g'(t)) \,\textrm{d}t
\end{equation*}
with $h(t,\xi):=\langle \xi, (t, g(t))-x(\xi)\rangle$.

Note that if $c$ is sufficiently small (say, $c<c_1/\pi$ with $c_1$
appearing in the statement of Lemma \ref{curv-in-ball}) Lemma
\ref{curv-in-ball} ensures that in the $t$-support of $\chi_1(t,
\xi)$ the curvature of the corresponding curve $\Gamma_g$ is $\asymp
K_{\xi}$. Hence $\partial_t (\chi_1(t, \xi))\lesssim K_{\xi}^2$ if
$1\leq |\vec{n}((t, g(t)))-\xi|/(c K_{\xi}^{-1})\leq 2$; $=0$
otherwise. The mean value theorem implies that
\begin{equation*}
\{t\in (-1, 1) : |\vec{n}((t, g(t)))-\xi|/(c K_{\xi}^{-1})\leq
2\}\subset B(t(\xi), 2A^{-1}c K_{\xi}^{-2})
\end{equation*}
for some constant $A$ (appearing in the statement of Lemma
\ref{curv-in-ball}). Also we get $|\partial_t h(t,\xi)|\gtrsim
K_{\xi}^{-1}$ if $|\vec{n}((t, g(t)))-\xi|/(c K_{\xi}^{-1}) \geq 1$.

For $I_2(\mu, \xi)$, after one integration by parts the boundary
terms for $t=\pm 1$ lead to the last bound in \eqref{asym-formula}
while the integral term is of size $O(K_{\xi}\mu^{-1})$.

To $I_1(\mu, \xi)$ we first apply a substitution
$u=K_{\xi}^{2}(t-t(\xi))$ and get
\begin{equation*}
I_1(\mu, \xi)=e^{-i\mu \langle \xi, x(\xi)\rangle}K_{\xi}^{-2}\int
e^{-i\mu h\left(t(\xi)+K_{\xi}^{-2}u,\xi\right)} \tau(u, \xi)
\,\textrm{d}u,
\end{equation*}
where
\begin{equation*}
\tau(u, \xi):=\chi_1\left(t(\xi)+K_{\xi}^{-2}u,
\xi\right)\left(\xi_2-\xi_1 g'(t(\xi)+K_{\xi}^{-2}u)\right)
\end{equation*}
and $u$-support of $\tau$ is contained in $B(0, 2A^{-1}c)$. By
Taylor's formula the phase function becomes
\begin{equation*}
h\left(t(\xi)+K_{\xi}^{-2}u,\xi\right)=K_{\xi}^{-3}\xi_2^{-2}u^2(1+\varepsilon(u,
\xi))/2,
\end{equation*}
where
\begin{equation*}
\varepsilon(u, \xi):=K_{\xi}^{-3}\xi_2^{3}u \int_0^1 g'''\left(
t(\xi)+K_{\xi}^{-2}us\right)(1-s)^2 \,\textrm{d}s.
\end{equation*}
Observe that $|\varepsilon(u, \xi)|\lesssim |u|$  in the $u$-support
of $\tau$. Hence $1/2\leq 1+\varepsilon(u, \xi)\leq 3/2$ if $c$ is
sufficiently small. Define $v=u(1+\varepsilon(u, \xi))^{1/2}$. If
$c$ is sufficiently small then $|\partial_u v|\asymp 1$ and
$|\partial_v u|\asymp 1$. Thus $|\partial_u^l v|\lesssim 1$ and
$|\partial_v^l u|\lesssim 1$. By this change of variable we get
\begin{equation*}
I_1(\mu, \xi)=e^{-i\mu \langle \xi, x(\xi)\rangle}K_{\xi}^{-2}\int
e^{-i\mu K_{\xi}^{-3}\xi_2^{-2}v^2/2} \tau(u(v), \xi)\partial_v u
\,\textrm{d}v.
\end{equation*}
Applying the method of stationary phase (say, H\"ormander's
\cite[Lemma 7.7.3]{hormander}) to the integral above yields the
leading term and the first error term on the right hand side of
\eqref{asym-formula}. This finishes the proof.
\end{proof}


\section{Estimate of Exponential Sums}\label{expsumsec}

Let $M>1$ and $T>0$ be parameters. In this section we consider
exponential sums of the form
\begin{equation*}
S(T, M; G, F)=\sum_{m\in\mathbb{Z}^2} G(m/M) e(TF(m/M)),
\end{equation*}
where $G:\mathbb{R}^2\rightarrow\mathbb{R}$ is $C^{\infty}$ smooth,
compactly supported, and bounded above by a constant, and
$F:\mathcal{D}\subset \mathbb{R}^2 \rightarrow\mathbb{R}$ is
$C^{\infty}$ smooth on an open convex domain $\mathcal{D}$ such that
\begin{equation*}
\textrm{supp} (G)\subset \mathcal{D} \subset c_0 B(0, 1),
\end{equation*}
where $c_0>0$ is a fixed constant.

The following lemma is M\"{u}ller's \cite[Lemma 1]{mullerII} (in a
slightly different form), namely the so-called iterated
one-dimensional Weyl--Van der Corput inequality.

\begin{lemma}\label{weyl-van-inequality}
Let $q\in \mathbb{N}$, $Q=2^q$, and $r_1, \ldots, r_q\in
\mathbb{Z}^2_*$ with $|r_i|\lesssim 1$. Furthermore, let $H$ be a
parameter which satisfies $1<H\lesssim M$. Set
$H_l=H_{q,l}=H^{2^{l-q}}$ for $l=1, \ldots, q$. Then
\begin{equation*}
\begin{split}
&|S(T, M; G, F)|^{Q} \\
&\quad\lesssim\frac{M^{2Q}}{H} + \frac{M^{2(Q-1)}}{H_1\cdot \ldots
\cdot H_q} \sum_{\substack{1\leq h_i<H_i \\1 \leq i \leq q}}
\left|S\left(\mathscr{H}TM^{-q}, M; G_q, F_q\right)\right|,
\end{split}
\end{equation*}
where $\mathscr{H}=\prod_{l=1}^{q}h_l$ and functions $G_q$ and $F_q$
are defined as follows:
\begin{equation*}
G_q(x)=G_q(x, h_1, \ldots, h_q)=\prod_{\substack{u_i\in\{0,1\}\\
1\leq i\leq q}} G\left(x+\sum_{l=1}^{q} \frac{h_l}{M} u_l r_l\right)
\end{equation*}
and
\begin{align*}
F_q(x)&=F_q(x, h_1, \ldots, h_q)\\
      &=\int_{(0,1)^q} \langle r_1, \nabla\rangle \cdots \langle r_q, \nabla\rangle F
       \left(x+\sum_{l=1}^{q} \frac{h_l}{M} u_l r_l\right) \, \textrm{d}u_1\ldots \textrm{d}u_q.
\end{align*}

The integral representation of $F_q$ is well defined on the open
convex set $\mathcal{D}_q=\mathcal{D}_q(h_1, \ldots, h_q)=\{x\in
\mathcal{D} : x+\sum_{l=1}^{q} (h_l/M) u_l r_l \in \mathcal{D}
\textrm{ for all } u_l\in \{0,1\}, l=1, \ldots, q\}$. Moreover,
$\supp(G_q)\subset \mathcal{D}_q\subset \mathcal{D}$.
\end{lemma}


For a set $E\subset \mathbb{R}^2$ and a positive number $r$, we
define $E_{(r)}$ to be the larger set
\begin{equation*}
E_{(r)}=\{x\in \mathbb{R}^2: \dist(E, x)<r \}.
\end{equation*}

By using the Weyl--Van der Corput inequality (A-process) and the
Poisson summation formula followed by the method of stationary phase
(B-process), we get the following estimate of $S(T, M; G, F)$.


\begin{proposition}\label{claim1} Let $q\in \mathbb{N}$, $Q=2^q$, and $K>1$ be a parameter. Assume that
\begin{equation}
\dist\left(\supp(G), \mathcal{D}^{c}\right)\gtrsim K^{-q-2}
\label{bdy-dist}
\end{equation}
and that for all $\nu\in \mathbb{N}_0^2$ and $y\in \mathcal{D}$,
\begin{equation}
D^{\nu}G(y)\lesssim K^{(q+2)|\nu|}, \label{aaa}
\end{equation}
\begin{equation}
D^{\nu}F(y)\lesssim \left\{ \begin{array}{ll}
1 & \textrm{if $0\leq |\nu|\leq 1$}\\
K^{|\nu|-3} & \textrm{if $|\nu|\geq 2$}
\end{array}\right.,  \label{upperforf}
\end{equation}
and
\begin{equation}
\left|\det\left( \nabla^2 D^{\alpha}F(y) \right)\right|\gtrsim
K^{-2}, \quad \alpha=(1, q-1). \label{lowerforf}
\end{equation}

If
\begin{equation}
M\geq K^{8q+2} \label{res-M}
\end{equation}
and
\begin{equation}
T\geq K^{4-(4q+2)/Q}M^{q-1+2/Q}, \label{restriction}
\end{equation}
then
\begin{equation}
S(T, M; G, F)\lesssim \left(K^{6q-1}TM^{6Q-q-6}\right)^{1/(3Q-2)}+R,
\label{kkk}
\end{equation}
where
\begin{equation*}
R=K^{(13q+3)/Q}M^{2-2/Q}\left(K^{1-6q}T^{-1}M^{q+2}
\right)^{1/(3Q-2)+\epsilon}
\end{equation*}
for any $\epsilon>0$.

The constant implicit in \eqref{kkk} depends only on $q$, $c_0$,
$\epsilon$, and constants implicit in \eqref{bdy-dist}, \eqref{aaa},
\eqref{upperforf}, and \eqref{lowerforf}.
\end{proposition}

\begin{remark}
One can write the bound \eqref{kkk} into a nicer form by removing
the error term $R$ at the cost of assuming (instead of
\eqref{restriction}) a more restricted condition on $T$. However we
keep the current form since it is good enough for later
applications.
\end{remark}

\begin{proof}[Proof of Proposition \ref{claim1}]

Let
\begin{equation}
1<H\leq c_3 K^{-2q-1}M    \label{con-on-H}
\end{equation}
with $c_3<1$ chosen (later) to be sufficiently small. Then $H\leq
M$.  We apply Lemma \ref{weyl-van-inequality} with $r_1=e_1$ and
$r_j=e_2$ ($j=2, \ldots, q$). Applying the Poisson summation formula
followed by a change of variables yields
\begin{align*}
S_1 &:=S(\mathscr{H}TM^{-q}, M; G_q, F_q) \\
&\ =\sum_{p\in \mathbb{Z}^2} K^{-2}M^2\int_{\mathbb{R}^2}
\Psi_q(z)e\left(\mathscr{H}TM^{-q}F_q(K^{-1} z)-K^{-1}M \langle p,
z\rangle\right)\,\textrm{d}z,
\end{align*}
where $\Psi_q(z)=G_q(K^{-1}z)$. It is obvious that
\begin{equation}
\supp(\Psi_q)\subset K\mathcal{D}_q \subset c_0 KB(0, 1).
\label{domain1}
\end{equation}
By \eqref{bdy-dist} we also have
\begin{equation}
\dist(\supp(\Psi_q), (K\mathcal{D}_q)^{c})\gtrsim K^{-q-1}.
\label{domain2}
\end{equation}

By the assumption \eqref{upperforf} there exists a constant $A_1$
such that
\begin{equation*}
|\nabla_z(F_q(K^{-1}z))|\leq (A_1/2) K^{q-3}.
\end{equation*}
We split $S_1$ into two parts
\begin{equation*}
S_1=\sum_{|p|<A_1K^{q-2}\mathscr{H}TM^{-q-1}}+\sum_{|p|\geq
A_1K^{q-2}\mathscr{H}TM^{-q-1}}=\textrm{:}S_2+R_1.
\end{equation*}

It is not hard to prove, by integration by parts (see for example
Hormander~\cite{hormander} Theorem 7.7.1), that
\begin{equation}
R_1\lesssim K^{3q+6}M^{-1}.\label{R5}
\end{equation}

Next we will estimate $S_2$. Define
$\lambda_1=K^{q-3}\mathscr{H}TM^{-q}$ and
\begin{equation*}
\Phi_q(z, p)=(\mathscr{H}TM^{-q}F_q(K^{-1} z)-K^{-1}M \langle p,
z\rangle)/\lambda_1.
\end{equation*}
Then
\begin{equation}
S_2=K^{-2}M^2 \sum_{|p|<A_1 K\lambda_1M^{-1}} \int
\Psi_q(z)e(\lambda_1 \Phi_q(z, p))\,\textrm{d}z. \label{S5}
\end{equation}

To estimate $S_2$ we discuss in two cases.

CASE 1. $\lambda_1\geq K^{8q+1}$.

For all $z\in K\mathcal{D}_q$, by \eqref{aaa}, \eqref{upperforf},
and \eqref{lowerforf}, we get
\begin{equation}
D^{\nu}_z\Psi_q(z)\lesssim K^{(q+1)|\nu|}, \label{ff}
\end{equation}
\begin{equation}
D^{\nu}_z \Phi_q(z, p)\lesssim 1,\quad  \textrm{if $|\nu|\geq 1$},
\label{gg}
\end{equation}
and
\begin{equation}
|\det\left(\nabla^2_{zz}\Phi_q(z, p) \right)|\gtrsim K^{-2q}.
\label{cc}
\end{equation}
To prove this lower bound \eqref{cc} we first note, by using the
definition of $F_q$ and the mean value theorem, that
\begin{equation*}
\frac{\partial^2}{\partial z_{l_1} \partial z_{l_2}} (\Phi_q(z,
p))=K^{1-q}\frac{\partial^{2} D^{\alpha}F}{\partial x_{l_1}
\partial x_{l_2}}(K^{-1}z)+O(K\frac{H}{M}).
\end{equation*}
The two terms on the right are $\lesssim$ $1$ and $c_3 K^{-2q}$
respectively (due to \eqref{upperforf} and \eqref{con-on-H}). Thus
\begin{equation*}
\det(\nabla^2_{zz}(\Phi_q(z,
p)))=K^{2-2q}\det(\nabla^2D^{\alpha}F(K^{-1}z))+O(c_3 K^{-2q}).
\end{equation*}
Then \eqref{cc} follows from \eqref{lowerforf} if $c_3$ is
sufficiently small.

With \eqref{domain1}, \eqref{domain2}, \eqref{ff}, \eqref{gg}, and
\eqref{cc}, we can estimate the integrals in $S_2$. Let us fix an
arbitrary $p\in \mathbb{Z}^2$ with $|p|<A_1 K\lambda_1M^{-1}$.

We first estimate the number of critical points of the phase
function $\Phi_q$. Denote $\widetilde{p}=K^{-1}Mp/\lambda_1$ and
$\mathfrak{F}(z)=K^{3-q}\nabla_z(F_q(K^{-1}z))$, then $\nabla_z
\Phi_q(z, p)=\mathfrak{F}(z)-\widetilde{p}$ and the critical points
are determined by the equation
\begin{equation*}
\mathfrak{F}(z)=\widetilde{p} \quad \textrm{for $z\in
K\mathcal{D}_q$}.
\end{equation*}
The bounds \eqref{gg} and \eqref{cc} imply that the mapping
$\mathfrak{F}=(\mathfrak{F}_1, \mathfrak{F}_2)$ satisfies
\begin{equation*}
D^{\nu}\mathfrak{F}_j(z)\lesssim 1   \quad \textrm{for $|\nu|\leq
2$, $j=1$, $2$},
\end{equation*}
and
\begin{equation*}
|\det(\nabla_z \mathfrak{F}(z))|\gtrsim K^{-2q}.
\end{equation*}
By \eqref{domain2}, we know that $\supp(\Psi_q)$ is strictly smaller
than $K\mathcal{D}_q$ and the distance between their boundary is
larger than $a_1 K^{-q-1}$ for some positive constant $a_1$. Let
$r_0=a_1 K^{-q-1}/2$. By Taylor's formula, there exists a positive
constant $a_2$ ($<a_1/2$) such that if $\tilde{z}$ is a critical
point in $(\supp(\Psi_q))_{(r_0)}$ then
\begin{equation}
|\nabla_z \Phi_q(z, p)|\gtrsim K^{-2q}|z-\tilde{z}|, \quad
\textrm{for any $z\in B(\tilde{z}, a_2 K^{-2q})$}. \label{taylor}
\end{equation}
Applying Lemma \ref{app:lemma:1} to $\mathfrak{F}$ with $r_0$ as
above yields two positive constants $a_3$ ($<a_2/2$) and $a_4$ such
that if $r_1=a_3 K^{-2q}$, $r_2=a_4 K^{-4q}$, then $\mathfrak{F}$ is
bijective from $B(z, 2r_1)$ to an open set containing
$B(\mathfrak{F}(z), 2r_2)$ for any $z\in (\supp(\Psi_q))_{(r_0)}$.
It follows, simply by a size estimate, that the number of critical
points in $(\supp(\Psi_q))_{(r_1)}$ is $\lesssim$
$K^2/r_1^{2}\lesssim K^{4q+2}$ .

For the $p$ that we have fixed, let $Z_j$ ($j=1, \ldots, J(p)$) be
all critical points in $(\supp(\Psi_q))_{(r_1)}$ of the phase
function $\Phi_q(z, p)$ and $\chi_j(z)=\chi_0((z-Z_j)/(c_4 r_1))$
with $c_4$ chosen below, where $\chi_0$ is a fixed smooth cut-off
function whose value is $1$ on $B(0, 1)$ and $0$ on the complement
of $B(0, 2)$.  Then the integral in $S_2$ can be decomposed as
\begin{equation}
\int \Psi_q(z)e(\lambda_1 \Phi_q(z, p))\,\textrm{d}z
=S_3+R_2,\label{S5-1}
\end{equation}
where
\begin{equation*}
S_3:=\sum_{j=1}^{J(p)}\int \chi_j(z) \Psi_q(z)e(\lambda_1 \Phi_q(z,
p))\,\textrm{d}z
\end{equation*}
and
\begin{equation*}
R_2:=\int \left(1-\sum_{j=1}^{J(p)}\chi_j(z)\right)
\Psi_q(z)e(\lambda_1 \Phi_q(z, p))\,\textrm{d}z.
\end{equation*}

It follows again from integration by parts (Theorem 7.7.1 in
\cite{hormander}) and also \eqref{taylor} that
\begin{equation}
R_2\lesssim K^{8qN+2}\lambda_1^{-N}, \quad N\in \mathbb{N}.
\label{boundR6}
\end{equation}

As for $S_3$, by Sogge and Stein's \cite[Lemma 2]{soggestein} (with
$\delta=K^{-2q}$) if $c_4$ is sufficiently small then for each $j$
we have
\begin{equation*}
\left| \int \chi_j(z) \Psi_q(z)e(\lambda_1 \Phi_q(z,
p))\,\textrm{d}z \right|\lesssim \lambda_1^{-1} K^{q}.
\end{equation*}
Hence
\begin{equation}
S_3\lesssim K^{5q+2}\lambda_1^{-1}. \label{boundS3}
\end{equation}

Noticing that we have assumed that $\lambda_1\geq K^{8q+1}$ in the
Case 1, it is easy to check that the bound \eqref{boundR6} of $R_2$
is less than the \eqref{boundS3} of $S_3$ if $N$ is sufficiently
large. Hence by using \eqref{S5}, \eqref{S5-1}, \eqref{boundR6}, and
\eqref{boundS3}, we get
\begin{align}
S_2&\lesssim K^{-2}M^2\left(\left(K\lambda_1M^{-1}\right)^2+1\right)K^{5q+2}\lambda_1^{-1} \nonumber  \\
   &\lesssim K^{5q+2}\lambda_1+K^{5q}M^2\lambda_1^{-1}.\label{case1}
\end{align}

CASE 2.  $\lambda_1< K^{8q+1}$.

Within this range of $\lambda_1$, the assumption \eqref{res-M}
implies $K\lambda_1M^{-1}<1$, hence the trivial estimate of $S_2$
(together with \eqref{domain1} and \eqref{ff}) yields
\begin{equation}
S_2\lesssim M^2\leq K^{8q+1}M^2 \lambda_1^{-1}. \label{case2}
\end{equation}

Combining the bounds of $S_2$ from Case 1 and 2 (namely,
\eqref{case1} and \eqref{case2}) yields
\begin{equation*}
S_2\lesssim K^{5q+2}\lambda_1+K^{8q+1}M^2 \lambda_1^{-1}.
\end{equation*}

By  \eqref{res-M} and  \eqref{R5} we have $R_1\lesssim K^{4-5q}\leq
1$, which is smaller than the above bound no matter whether
$\lambda_1\geq 1$ or $<1$. Thus
\begin{equation*}
S_1=S_2+R_1\lesssim
K^{6q-1}\mathscr{H}TM^{-q}+K^{7q+4}(\mathscr{H}T)^{-1}M^{q+2},
\end{equation*}
where we have already used the definition of $\lambda_1$.

Plugging this bound of $S_1$ into the inequality in Lemma
\ref{weyl-van-inequality} gives
\begin{equation}
|S(T, M; G, F)|^{Q}\lesssim
M^{2Q}H^{-1}+K^{6q-1}TM^{2Q-q-2}H^{2-2/Q}+\textrm{E}, \label{cccc}
\end{equation}
where
\begin{equation*}
\textrm{E}=K^{7q+4}T^{-1}M^{2Q+q}H^{-2+2/Q}(\log H)^q.
\end{equation*}

In order to balance the first two terms on the right side of
\eqref{cccc} we let
\begin{equation*}
H=c_3(K^{1-6q}T^{-1}M^{q+2})^{Q/(3Q-2)}.
\end{equation*}
This choice of $H$ satisfies \eqref{con-on-H}. Indeed, the
assumption \eqref{restriction} implies $H\leq c_3 K^{-2q-1} M$. We
also have $1<H$ since we can assume $T<c_5 K^{1-6q}M^{q+2}$ with a
sufficiently small $c_5$ (otherwise the trivial bound of $S(T, M; G,
F)$, {\it i.e.} $M^2$, is better than \eqref{kkk}).

With this $H$ the bound \eqref{cccc} leads to the desired bound
\eqref{kkk}.
\end{proof}


\section{Nonvanishing Determinants}\label{nonvanishing-det}

For $\xi\in\mathfrak{C}_0$ let $H(\xi):=\langle \xi, x(\xi)\rangle$,
which is obviously positively homogeneous of degree one, {\it i.e.}
$H(\lambda \xi)=\lambda H(\xi)$ if $\lambda>0$.

In this section we provide some technical results related to $H$.
They are similar to those proved in Section 3 of \cite{guo2} and
\cite{guo3}. We omit the proof of Lemma \ref{upperforH} and provide
a sketch for Lemma \ref{non-vanishinglemma} for completeness. In
particular Lemma \ref{non-vanishinglemma} is a refinement to
M\"{u}ller's \cite[Lemma 3]{mullerII} in two dimensional case. It is
obtained based on M\"{u}ller's original proof. We refer interested
readers to these references for more details.

By using implicit differentiation and induction we can easily get
the following lemma.

\begin{lemma}\label{upperforH}
$H$ is smooth in $\mathfrak{C}_0$ and for every $\xi\in
\mathfrak{C}_0$
\begin{equation*}
H(\xi)\lesssim |\xi|,
\end{equation*}
\begin{equation*}
D^{\nu}H(\xi) \lesssim 1 \quad \textrm{for $|\nu|=1$},
\end{equation*}
and
\begin{equation*}
D^{\nu}H(\xi) \lesssim_{|\nu|} |\xi|^{1-|\nu |}K_{\xi}^{|\nu|-3}
\quad \textrm{for } |\nu|\geq 2.
\end{equation*}
Furthermore the matrix $\nabla^2_{\xi \xi}H(\xi)$ has two
eigenvalues $0$ and $(|\xi|K_{\xi})^{-1}$.
\end{lemma}

Given vectors $v_1$, $v_2\in \mathbb{R}^2$, by writing $V=(v_1,
v_2)$ we mean $V$ is the matrix with column vectors $v_1, v_2$. For
$q\in \mathbb{N}$, real $u_1$ and $u_2$, and $y\in\mathfrak{C}_0$,
define
\begin{equation*}
F(u_1, u_2)=H(y+u_1 v_1+u_2 v_2)
\end{equation*}
and
\begin{equation*}
h_q(y, v_1, v_2)=\det\left(g_{i,j}(y, v_1, v_2)\right)_{1\leq i,
j\leq 2}
\end{equation*}
with
\begin{equation*}
g_{i,j}(y, v_1, v_2)=\frac{\partial^{q+2}F}{\partial u_1
\partial u_i
\partial u_j \partial u_2^{q-1}}(0) .
\end{equation*}

The following lemma provides nonvanishing determinants needed for
our  application of the method of stationary phase. This result can
be easily extended to non-unit vectors by using the homogeneity of
$H$.

\begin{lemma} \label{non-vanishinglemma}
For every $\xi\in S^1 \cap \mathfrak{C}_0$ there exist two
orthogonal vectors $v_1(\xi)$, $v_2(\xi)\in \mathbb{Z}^2$ such that
\begin{equation}
|v_1|=|v_2|\asymp_q K_{\xi}^{q+1}\quad \textrm{and}\quad \|(v_1,
v_2)^{-1}\|\lesssim_q K_{\xi}^{-q-1}, \label{lemma4:4-1}
\end{equation}
and a positive constant $c_2=c_2(q)$ such that for any $\eta \in
B(\xi, c_2K_\xi^{-q-2})$ we have
\begin{equation}
|h_q(\eta, v_1, v_2)|\gtrsim_q K_\xi^{2q^2+6q+2}. \label{bbb}
\end{equation}
\end{lemma}

\begin{proof}
Let us arbitrarily fix a point $\xi=(\xi_1, \xi_2)^{t}\in S^1 \cap
\mathfrak{C}_0$. Denote $v_1^*=(-\xi_2, \xi_1)^{t}$ and
$v_2^*=(\xi_1, \xi_2)^{t}$. Then
\begin{equation}
h_q(\xi, v_1^*, v_2^*)=-q!^2 K_{\xi}^{-2}. \label{ddd}
\end{equation}
This equality follows from Lemma \ref{upperforH} and the homogeneity
of $H$ which actually imply that $g_{2,2}(\xi, v_1^*, v_2^*)=0$,
\begin{equation*}
g_{1,2}(\xi, v_1^*, v_2^*)=g_{2,1}(\xi, v_1^*, v_2^*)=(-1)^q
q!K_{\xi}^{-1},
\end{equation*}
and
\begin{equation*}
g_{1,1}(\xi, v_1^*, v_2^*)\lesssim_q 1.
\end{equation*}
Check the proof of \cite[Lemma 3]{mullerII} and \cite[Lemma
3.4]{guo2} for details.

For any integer $N\geq 2\sqrt{2}$ and $1\leq l\leq 2$ there exist
$N_l\in\mathbb{Z}$ such that $|\xi_l-N_l/N|\leq 1/N$. Denote
$v_1=(-N_2, N_1)^{t}$, $v_2=(N_1, N_2)^{t}$, and
$\widetilde{v}_l=v_l/N$. Then $|v_l^*-\widetilde{v}_l|\leq
\sqrt{2}/N$ and $1/2 \leq|\widetilde{v}_l|\leq 3/2$.

Assume that $N$ is the smallest integer not less than $A
K_\xi^{q+1}$ and that $\eta\in B(\xi, c_2 K_\xi^{-q-2})$ with $A$
and $c_2(\leq c_1)$ both chosen below, where $c_1$ is the constant
appearing in Lemma \ref{curv-in-ball}. By the mean value theorem and
Lemma \ref{curv-in-ball} and \ref{upperforH} we get
\begin{equation*}
\left|g_{i,j}(\xi, v_1^*, v_2^*)-g_{i,j}(\eta, \widetilde{v}_1,
\widetilde{v}_2)\right|\lesssim_q K_\xi^{q-1}N^{-1}+c_2 K_\xi^{-2},
\end{equation*}
which leads to
\begin{equation}
|h_q(\xi, v_1^*, v_2^*)-h_q(\eta, \widetilde{v}_1,
\widetilde{v}_2)|\leq C(q) \left(K_\xi^{q-1}N^{-1}+c_2
K_\xi^{-2}\right)\label{eee}
\end{equation}
for some constant $C(q)$. Hence if $A$ is sufficiently large and
$c_2$ is sufficiently small then \eqref{ddd} and \eqref{eee} imply
that
\begin{equation*}
|h_q(\eta, \widetilde{v}_1, \widetilde{v}_2)|\geq q!^2 K_\xi^{-2}/2.
\end{equation*}

The desired bound \eqref{bbb} then follows from the equality
\begin{equation*}
|h_q(\eta, v_1, v_2)|=N^{2(q+2)}|h_q(\eta, \widetilde{v}_1,
\widetilde{v}_2)|.
\end{equation*}
And $v_1$ and $v_2$ are the desired vectors satisfying
\eqref{lemma4:4-1}.
\end{proof}



\section{The Associated Lattice Point Problem}\label{lattice-point}

In this section we combine the tools from previous sections to prove
Theorem \ref{lattice-thm}.

As a preliminary step we transform the problem of counting lattice
points into an analytical one by a standard procedure: we consider
instead smoothed versions of $N_{\Omega}(\mu)$ (with errors under
control) to which we apply the Poisson summation formula; in each
resulting sum we separate the first term with the rest and
then use the Green's formula to get certain integrals along
boundaries. See for example Colin de Verdi\`{e}re~\cite{colin}
Section 2 for this procedure. Notice that Colin de Verdi\`{e}re's
enlarged domains $G_{\mu, \varepsilon}^{\pm}$ are formed by a
combination of dilation and translation, which are ingeniously
suitable for the comparison between $N_{\Omega}(\mu)$ and its
smoothed analogues.

For the completeness we carry out the above procedure in the proof
of the following lemma. A small difference is, unlike the treatment
of Colin de Verdi\`{e}re, that ours considers the lattice counting
function $N_{\Omega}(\mu)$ directly without adding a weight.

\begin{lemma}\label{1-reduction}
Let $\rho\in C_0^{\infty}(\mathbb{R}^2)$ be a nonnegative, radial,
and real-valued function such that $\int\rho=1$ and $\supp
\rho\subset B(0,1)$. Then
\begin{equation}
\begin{split}
&\left|N_{\Omega}(\mu)-|\Omega|\mu^2+\mu/2\right|\lesssim
\varepsilon^{-1}+\mu \varepsilon+\\
&\qquad \qquad \qquad \mu\left|\sum_{k\in \mathfrak{C}_0}
|k|^{-1}I(2\pi \mu |k|, k/|k|)\widehat{\rho}(\varepsilon k)e^{-2\pi
i (1/4\pm 2\varepsilon)k_2} \right|,
\end{split}\label{nnn}
\end{equation}
where the cone $\mathfrak{C}_0$ and the integral $I$ are defined by
\eqref{key-cone} and \eqref{key-integral} respectively and the
implicit constant only depends on the domain $\Omega$.
\end{lemma}

\begin{proof}
Let $\chi_{G_{\mu, \varepsilon}^{+}}$ be the characteristic function
of the domain
\begin{equation*}
\{(x_1, x_2)\in \mathbb{R}^2 : |x_1|\leq \mu, \max(0,-x_1)\leq
x_2\leq \mu g(x_1/\mu)+2\varepsilon\}
\end{equation*}
and $\chi_{G_{\mu, \varepsilon}^{-}}$ be the difference of two
characteristic functions, namely
\begin{align*}
\chi_{G_{\mu, \varepsilon}^{-}}=&\textbf{1}_{\{(x_1, x_2)\in
\mathbb{R}^2 : |x_1|\leq \mu, \max(0,-x_1)\leq x_2\leq \mu
g(x_1/\mu)-2\varepsilon\}}\\
&-\textbf{1}_{\{(x_1, x_2)\in \mathbb{R}^2 : |x_1|\leq \mu, \mu
g(x_1/\mu)-2\varepsilon\leq x_2\leq \max(0,-x_1) \}}.
\end{align*}

It is then geometrically evident that
\begin{equation*}
\sum_{k\in \mathcal{R}} \chi_{G_{\mu,
\varepsilon}^{-}}*\rho_{\varepsilon}(k)\leq N_{\Omega}(\mu)\leq
\sum_{k\in \mathcal{R}}\chi_{G_{\mu,
\varepsilon}^{+}}*\rho_{\varepsilon}(k),
\end{equation*}
where
$\rho_{\varepsilon}(x)=\varepsilon^{-2}\rho(\varepsilon^{-1}x)$.
Applying the Poisson summation formula yields
\begin{equation}
R_{\varepsilon}^{-}(\mu)-C\mu\varepsilon\leq
N_{\Omega}(\mu)-|\Omega|\mu^2 \leq
R_{\varepsilon}^{+}(\mu)+C\mu\varepsilon,\label{hhh}
\end{equation}
where $C$ is a constant ({\it i.e.} the length of the curve
$\wideparen{P_1 P_2}$) and       
\begin{equation*}
R_{\varepsilon}^{\pm}(\mu):=\sum_{k\in \mathbb{Z}^2_*}
\widehat{\chi}_{G_{\mu,
\varepsilon}^{\pm}}(k)\widehat{\rho}(\varepsilon k)e^{-\pi i k_2/2}.
\end{equation*}

We study $R_{\varepsilon}^{+}(\mu)$ below while
$R_{\varepsilon}^{-}(\mu)$ can be handled similarly. By Green's
formula we have
\begin{equation*}
\widehat{\chi}_{G_{\mu, \varepsilon}^{+}}(k)=(2\pi i|k|)^{-1}
\oint_{\partial G_{\mu, \varepsilon}^{+}}e^{-2\pi i \langle k,
x\rangle}\left(\frac{k_2}{|k|}
\,\textrm{d}x_1-\frac{k_1}{|k|}\,\textrm{d}x_2 \right).
\end{equation*}
Notice that $\partial G_{\mu, \varepsilon}^{+}$ can be decomposed
into five parts: four line segments (including two short vertical
and two long ones) and one curly curve $\mu\wideparen{P_1 P_2}+(0,
2\varepsilon)$. Accordingly we can split the sum
$R_{\varepsilon}^{+}(\mu)$ into five.

Integration by parts shows that the two sums associated with two
vertical line segments are both of size $O(\varepsilon^{-1})$. The
other two associated with long line segments can be evaluated more
accurately--each one is equal to $-\mu/4+O(\varepsilon^{-1})$.

With these bounds, \eqref{hhh} leads to
\begin{align*}
&\left|N_{\Omega}(\mu)-|\Omega|\mu^2+\mu/2\right|\lesssim
\varepsilon^{-1}+\mu
\varepsilon+\\
&\qquad \qquad \qquad \mu\left|\sum_{k\in \mathbb{Z}^2_*}
|k|^{-1}I(2\pi \mu |k|, k/|k|)\widehat{\rho}(\varepsilon k)e^{-2\pi
i (1/4\pm 2\varepsilon)k_2} \right|.
\end{align*}

We split the last sum depending on whether $|k|\in \mathfrak{C}_0$
or $|k|\in \mathbb{Z}^2_*\setminus \mathfrak{C}_0$.

For the former part we are allowed to restrict $k$ to the first
quadrant due to symmetry.

The contribution of the latter part is
$O(\mu^{1/3}+\varepsilon^{-1})$. (The $O(\mu^{1/3})$ term can be
ignored since $\mu^{1/3}<\varepsilon^{-1}+\mu \varepsilon$.) This
estimate follows from integration by parts. We discuss in three
cases as follows.

If $k=(0, k_2)$, $k_2\in \mathbb{Z}_*$, then $I(2\pi \mu |k|,
k/|k|)\lesssim (\mu |k_2|)^{-2/3}$. (To prove this one can split the
integral $I$ into two parts over $[-1, 1-\delta]$ and $[1-\delta,
1]$, estimate them by integration by parts and trivial estimate
respectively, then balance the resulting bounds by choosing
$\delta=(\mu |k_2|)^{-2/3}$.) Therefore these $k$'s produce an
$O(\mu^{1/3})$ bound.

If $k=(k_2, k_2)$, $k_2\in \mathbb{Z}_*$, the same argument as above
proves an $O(\mu^{1/3})$ bound.

We are then left with the case $k\in \mathbb{Z}^2_*$ with
$-\infty<k_2/k_1<1$. For these points the phase function of $I(2\pi
\mu |k|, k/|k|)$ has no critical point. We can integrate by parts
directly. Thus
\begin{equation*}
|I(2\pi \mu |k|, k/|k|)|\lesssim \frac{1}{\mu
|k|}\left(\left|\frac{k_2}{k_1}\right|+\left|\frac{k_1+k_2}{k_1-k_2}\right|+\frac{|k|^2}{k_1(k_1-k_2)}\right).
\end{equation*}
Using this bound we get an $O(\varepsilon^{-1})$ bound. This
concludes the proof.
\end{proof}

\begin{remark}
One can easily get that $P_{\Omega}(\mu)=-\mu/2+O(\mu^{2/3})$ by
using Lemma \ref{1-reduction} and Proposition
\ref{estimate-key-integral} and \ref{asym-R^2} for all integral
points in $\mathfrak{C}_0$.

If we want a bound better than $O(\mu^{2/3})$ the estimation is
harder. Our idea is to decompose the cone $\mathfrak{C}_0$ into
different parts (see Figure \ref{figure2}) depending on a parameter
$\Delta$ and to apply Proposition \ref{estimate-key-integral} for
$D_1$ and $D_3$ and the AB-process (Proposition \ref{claim1}) for
$D_2$. We gain in the latter part, balance the bound with that of
the former part by choosing a proper $\Delta$, and then get a better
bound as a result.
\end{remark}

\begin{figure}
\includegraphics[width=0.35\textwidth]{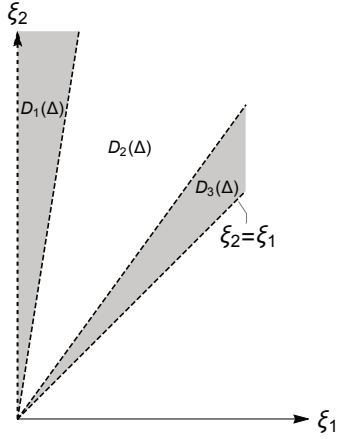}\caption{A decomposition of $\mathfrak{C}_0$}\label{figure2}
\end{figure}

\begin{proof}[Proof of Theorem \ref{lattice-thm}]
Let $\Delta>0$ be a parameter. (We will in fact choose
$\Delta=\mu^{1/495}$ and $\varepsilon=\mu^{-1/3-1/495}$ at the end.)
We decompose $\mathfrak{C}_0$ into three disjoint cones depending on
$\Delta$ (see Figure \ref{figure2}), namely
\begin{equation*}
\mathfrak{C}_0=D_1(\Delta)\cup D_2(\Delta) \cup D_3(\Delta),
\end{equation*}
where
\begin{equation*}
D_2(\Delta)=\{\xi\in \mathfrak{C}_0: K_{\xi}\leq \Delta\}
\end{equation*}
and $D_1(\Delta)$ ($D_3(\Delta)$) is a thin cone on the left
(right). Lemma \ref{angle-curv} shows that angles of $D_1$ and $D_3$
are both $\asymp 1/\Delta$.

We apply Proposition \ref{estimate-key-integral} and Lemma
\ref{angle-curv}, compare the sums with integrals in polar
coordinates, and get
\begin{align}
&\mu\left|\sum_{k\in D_1(\Delta)} |k|^{-1}I(2\pi \mu |k|,
k/|k|)\widehat{\rho}(\varepsilon k)e^{-2\pi i (1/4\pm
2\varepsilon)k_2} \right|\nonumber \\
&\quad \lesssim \sum_{k\in D_1(\Delta)}
\left(\mu^{1/2}|k|^{-3/2}K_{k}^{-1/2}
+\mu^{1/3}|k|^{-5/3}\right)|\widehat{\rho}(\varepsilon
k)|\nonumber\\
&\quad \lesssim
\Delta^{-3/2}\mu^{1/2}\varepsilon^{-1/2}+\Delta^{-1}\mu^{1/3}\varepsilon^{-1/3}.\label{D1D3}
\end{align}
Note that the first term of \eqref{D1D3} is larger than the second
one if $\Delta \leq \mu^{1/3}\varepsilon^{-1/3}$.

Since $D_3$ is similar with $D_1$, we only treat $D_2$ below.
Applying Proposition \ref{asym-R^2} yields
\begin{equation}
\mu \sum_{k\in D_2(\Delta)} |k|^{-1}I(2\pi \mu |k|,
k/|k|)\widehat{\rho}(\varepsilon k)e^{-2\pi i (1/4\pm
2\varepsilon)k_2}=A S+E_1,\label{equ6:1}
\end{equation}
where $A$ is an absolute constant,
\begin{equation*}
S:=\mu^{1/2} \sum_{k\in D_2(\Delta)} |k|^{-3/2}
K_{k}^{-1/2}\widehat{\rho}(\varepsilon k)e^{-2\pi i \mu H_1(k)},
\end{equation*}
\begin{equation*}
H_1(k):=\langle k, x(k)\rangle+(1/4\pm 2\varepsilon)k_2/\mu,
\end{equation*}
and
\begin{equation}
E_1=O\left(\Delta^{5/2}\mu^{-1/2}+\Delta \log
(\varepsilon^{-1})+\varepsilon^{-1}\right).\label{equ6:2}
\end{equation}

We next introduce a partition of unity associated with
$\mathfrak{C}_0$, which will be used to split the sum $S$. For each
$\xi \in S^{1}\cap \mathfrak{C}_0$ there exists a cone (contained in
$\mathfrak{C}_0$)
\begin{equation*}
\mathfrak{C}(\xi, 2r(\xi)):=\mathop{\cup}\limits_{l>0} l B(\xi,
2r(\xi)),
\end{equation*}
where $r(\xi)=c_2K_\xi^{-q-2}/2$ and $c_2$ is the constant appearing
in the statement of Lemma \ref{non-vanishinglemma}. Note that Lemma
\ref{curv-in-ball} implies that $K_\eta \asymp K_\xi$ if $\eta\in
\mathfrak{C}(\xi, 2r(\xi))$. From the family of cones
$\{\mathfrak{C}(\xi, r(\xi)/2): \xi \in S^{1}\cap \mathfrak{C}_0\}$
we can choose, by a Vitali procedure, a sequence
$\{\mathfrak{C}(\xi_i, r(\xi_i)/2)\}_{i=1}^{\infty}$ such that these
cones still cover $\mathfrak{C}_0$ and that $\{\mathfrak{C}(\xi_i,
r(\xi_i))\}_{i=1}^{\infty}$ satisfies the bounded overlap property.
Denote
\begin{equation*}
\mathfrak{C}_i=\mathfrak{C}(\xi_i, r(\xi_i)).
\end{equation*}
Then the collection $\{\mathfrak{C}_i\}_{i=1}^{\infty}$ forms an
open cover of $\mathfrak{C}_0$. We can construct a partition of
unity $\{\psi_i\}_{i=1}^{\infty}$ such that
\begin{enumerate}[\upshape (i)]
\item  $\sum_{i}\psi_i \equiv 1$ on $\mathfrak{C}_0$ and $\psi_i\in C_0^{\infty}(\mathfrak{C}_i)$;
\item  each $\psi_i$ is positively homogeneous of degree zero;
\item  $|D^{\nu}\psi_i|\lesssim_{|\nu|} K_{\xi_i}^{(q+2)|\nu|}$ on $\mathscr{C}_1:=\{\xi \in \mathbb{R}^2: 1/2 \leq |\xi| \leq 2 \}$.
\end{enumerate}

From the family $\{\mathfrak{C}_i\}_{i=1}^{\infty}$ we can find a
subfamily $\{\mathfrak{C}_i\}_{i\in \mathscr{A}}$ which covers
$D_2(\Delta)$, where $\mathscr{A}=\mathscr{A}(\Delta)$ is an index
set such that $i\in \mathscr{A}$ if and only if $\mathfrak{C}_i$
intersects $D_2(\Delta)$. Since $r(\xi_i)\gtrsim \Delta^{-q-2}$ for
any $i\in \mathscr{A}$, a size estimate gives that
$\#\mathscr{A}\lesssim \Delta^{q+2}$.

We also introduce a standard dyadic decomposition. Let $\varphi \in
C_0^{\infty}(\mathscr{C}_1)$ be a real nonnegative radial function
such that$\sum_j \varphi(\xi/2^j)\equiv 1$ if $\xi \neq 0$.

Then we can split the sum $S$ as follows:
\begin{equation}
S=\mu^{1/2}\sum_{i\in \mathscr{A}}\sum_{j\in \mathbb{N}} S_{i,
j}+O(\Delta^{-3/2}\mu^{1/2}\varepsilon^{-1/2}),\label{equ6:3}
\end{equation}
where
\begin{equation*}
S_{i, j}:=\sum_{k\in \mathbb{Z}^2}
\psi_i(2^{-j}k)\varphi(2^{-j}k)|k|^{-3/2}K_{k}^{-1/2}\widehat{\rho}(\varepsilon
k)e^{-2\pi i \mu H_1(k)}.
\end{equation*}
The last error term in \eqref{equ6:3} comes out due to the small
difference between $D_2(\Delta)$ and the support of $\sum_{i\in
\mathscr{A}}\psi_i$. Its estimate is just like \eqref{D1D3}.

For any fixed $i\in \mathscr{A}$ and $q\in \mathbb{N}$ we use the
results in Section \ref{nonvanishing-det} with $H$ replaced by
$H_1$. This essentially changes nothing since these two functions
only differ by a linear term containing a uniformly bounded
coefficient. By Lemma \ref{non-vanishinglemma} and the homogeneity
of $H_1$ there exist two orthogonal integral vectors $v_1(\xi_i)$
and $v_2(\xi_i)$ such that
\begin{equation}
|h_q(\eta, v_1, v_2)|\gtrsim K_{\xi_i}^{2q^2+6q+2} \quad \textrm{if}
\ \eta \in \mathop{\cup}\limits_{1/4\leq l\leq 4} l B(\xi_i,
2r(\xi_i)).\label{determinant}
\end{equation}
Let $L=[\mathbb{Z}^2:\mathbb{Z}v_1\oplus\mathbb{Z}v_2]$ be the index
of the lattice spanned by $v_1$ and $v_2$ in the lattice
$\mathbb{Z}^2$. Then there exist vectors $b_l\in \mathbb{Z}^2$
($l=1, \ldots, L$) such that
\begin{equation*}
\mathbb{Z}^2=\biguplus_{l=1}^{L}(\mathbb{Z}v_1+\mathbb{Z}v_2+b_l).
\end{equation*}
It follows from Lemma \ref{non-vanishinglemma} that $L\lesssim
K_{\xi_i}^{2q+2}$ and $|b_l|\lesssim K_{\xi_i}^{q+1}$.

Let $N, j\in \mathbb{N}$ be arbitrarily fixed. Write
$k=m_1v_1+m_2v_2+b_l$. Then
\begin{equation}
|S_{i,
j}|=\left|K_{\xi_i}^{-1/2}2^{-3j/2}(1+2^j\varepsilon)^{-N}\sum_{l=1}^{L}S(T,
M; G_l, F_l)\right|,\label{formula2}
\end{equation}
where $T=\mu 2^j$, $M=K_{\xi_i}^{-q-1}2^j$,
\begin{equation*}
F_l(y)=H_1\left(2^{-j}(Mv_1y_1+Mv_2y_2+b_l)\right),
\end{equation*}
and
\begin{equation*}
G_l(y)=K_{\xi_i}^{1/2}2^{3j/2}(1+2^j\varepsilon)^{N} U_{i,
j}(Mv_1y_1+Mv_2y_2+b_l)
\end{equation*}
with
\begin{equation*}
U_{i,
j}(k)=\psi_i(2^{-j}k)\varphi(2^{-j}k)|k|^{-3/2}K_{k}^{-1/2}\widehat{\rho}(\varepsilon
k).
\end{equation*}
We consider the function $F_l$ over the convex domain
\begin{equation*}
\mathcal{D}_l=\{y\in \mathbb{R}^2 : 2^{-j}(Mv_1y_1+Mv_2y_2+b_l)\in
\mathop{\cup}\limits_{1/4\leq l\leq 4} l B(\xi_i, 2r(\xi_i)) \}.
\end{equation*}
The support of $G_l$ satisfies
\begin{equation*}
\supp(G_l)\subset \{y\in \mathbb{R}^2 :
2^{-j}(Mv_1y_1+Mv_2y_2+b_l)\in \overline{\mathscr{C}_1\cap
\mathfrak{C}_i} \}\subset \mathcal{D}_l.
\end{equation*}

We apply to $S(T, M; G_l, F_l)$ Proposition \ref{claim1} (with
$G=G_l$, $F=F_l$, $K=K_{\xi_i}$, $\mathcal{D}=\mathcal{D}_l$, and
$q=3$) when $C\Delta^{30}\leq 2^j\leq C'\mu^{4/5}$ for some proper
constants $C$ and $C'$ such that the assumptions of Proposition
\ref{claim1} are satisfied, which can be verified by using Lemma
\ref{non-vanishinglemma}, \eqref{determinant}, and the following
facts: if $K_{\xi_i}^{q+1}\leq 2^j$ then $\mathcal{D}_l\subset c_0
B(0, 1)$ for some constant $c_0$ (depending only on $q$);
\begin{equation*}
\dist \left(\left(\mathop{\cup}\limits_{1/4\leq l\leq 4} l B(\xi_i,
2r(\xi_i)) \right)^c \, , \, \overline{\mathscr{C}_1\cap
\mathfrak{C}_i}\right)\geq c_2 K_{\xi_i}^{-q-2}/8;
\end{equation*}
and
\begin{equation*}
D^{\nu}U_{i, j} \lesssim
K_{\xi_i}^{-1/2+(q+2)|\nu|}2^{-(3/2+|\nu|)j}(1+2^j\varepsilon)^{-N}.
\end{equation*}

Depending on the size of $2^j$ we split \eqref{equ6:3} as follows:
\begin{align}
S&=\mu^{1/2}\sum_{C\Delta^{30}\leq 2^j\leq C'\mu^{4/5}}\sum_{i\in
\mathscr{A}}S_{i, j} +O(\Delta^{-3/2}\mu^{1/2}\varepsilon^{-1/2})\label{lll}\\
&\quad+\mu^{1/2}\left(\sum_{2^j<C\Delta^{30}}+\sum_{2^j>C'\mu^{4/5}}\right)\sum_{i\in
\mathscr{A}}S_{i, j}.\label{mmm}
\end{align}
By \eqref{formula2}, Proposition \ref{claim1}, $1\lesssim
K_{\xi_i}\lesssim \Delta$, and bounds of $\#\mathscr{A}$ and $L$, we
get that the double sum on the right side of \eqref{lll} is bounded
by
\begin{equation}
\Delta^{68/11}\mu^{1/2+1/22}\varepsilon^{-7/22}+\Delta^{399/44}\mu^{5/11}\varepsilon^{-19/44-4\epsilon}.\label{ooo}
\end{equation}

Balancing $\Delta^{68/11}\mu^{1/2+1/22}\varepsilon^{-7/22}$ with
$\Delta^{-3/2}\mu^{1/2}\varepsilon^{-1/2}$ (in \eqref{D1D3})
yields
\begin{equation*}
\Delta=(\mu\varepsilon^4)^{-1/169}.
\end{equation*}
With this choice of $\Delta$ we continue to balance
$\Delta^{68/11}\mu^{1/2+1/22}\varepsilon^{-7/22}$ with
$\mu\varepsilon$ in \eqref{nnn} and get
$\varepsilon=\mu^{-1/3-1/495}$ and $\Delta=\mu^{1/495}$. This
readily leads to the desired bound stated in Theorem
\ref{lattice-thm} after we verify (by simple calculus) that the
bounds $\Delta^{-1}\mu^{1/3}\varepsilon^{-1/3}$ in \eqref{D1D3},
$E_1$ in \eqref{equ6:2},
$\Delta^{399/44}\mu^{5/11}\varepsilon^{-19/44-4\epsilon}$ in
\eqref{ooo}, and the sums in \eqref{mmm} are all $\lesssim
\mu^{2/3-1/495}$.
\end{proof}


\section{From Lattice Points Counting to Eigenvalue Counting}\label{spec-to-lattice}

In this section we prove Theorem \ref{spectrum-thm}. We mainly follow Y. Colin de Verdi\`{e}re's argument in \cite{colin}.

By separating variables, one can easily find all the eigenvalues of the Laplace operator associated with the unit disk, which are
\[\sigma_D=\{(j_{n,k})^2: n\in\mathbb{Z} ,\,k=1,2,\cdots\}\]
for the Dirichlet boundary condition, and
\[\sigma_N=\{(j^{\prime}_{n,k})^2: n\in\mathbb{Z} ,\,k=1,2,\cdots\}\]
for the Neumann boundary condition, where  $j_{n,k}$ is the $k$th  zero of the Bessel function $J_n$, and $j^{\prime}_{n,k}$ is the $k$th zero of its derivative $J^{\prime}_n$. Note that the Bessel function has the  symmetry $J_{-n}(x)=(-1)^nJ_n(x)$, which implies $j_{-n,k}=j_{n,k}$ and $j'_{-n,k}=j'_{n,k}$.

To describe the asymptotic behavior of large $j_{n,k}$'s, Y. Colin de Verdi\`{e}re introduced a function $F$ defined on the cone
 $S=\{(t,s)\in\mathbb{R}^2: s\ge\max(0,-t)\}$, which is characterized by
\begin{enumerate}
\item[(a)] $F: S \to \mathbb R$ is homogeneous of degree 1,
\item[(b)] $F \equiv 1$ on  $\mathrm{graph}(g)$ (which is a section of $S$, see Figure \ref{figure1.1}).
\end{enumerate}
As observed by Colin de Verdi\`ere, the cone $S$ admits the following very nice ``involution invariance",
\[(x,y-x-1/4) \in S \Longleftrightarrow (-x,y-1/4) \in S.\]
Moreover, the function $F$ is also invariant under such involution:
\begin{lemma}\label{F-nk}
For any $(x, y-x-1/4) \in S$, one has 
\[F(x, y-x-1/4)=F(-x, y-1/4).\] 
\end{lemma}
\begin{proof}
Set $l=F(-x, y-1/4)$. Then by the definition of $F$, 
\[
g(-x/l)=(y-1/4)/l.
\]
It follows from definition of $g$ that 
\[
\frac{y-1/4}l = \frac 1\pi \left(\sqrt{1-(\frac xl)^2}+\frac {x}l (\pi-\arccos \frac {x}l) \right) = g(\frac xl)+\frac xl,
\]
i.e. $g(x/l) = (y-x-1/4)/l$. This implies $F(x,y-x-1/4)=l$. 
\end{proof}
Applying the method of stationary phase to the integral representation of the Bessel function one can prove
\begin{lemma}[\cite{colin}]\label{sta}
If $j_{n,k}>(1+c_0)|n|$ for some constant $c_0$, then
\[
 j_{n,k}=\left\{
\begin{array}{ll}
F(n,k-\frac{1}{4})+O(\frac{1}{n+k+1}),\;&\mathrm{if}\;n\ge 0,\\
F(n,k+|n|-\frac{1}{4})+O(\frac{1}{|n|+k+1}),\;&\mathrm{if}\;n<0.
\end{array}
\right.
\]
\end{lemma} 

Thanks to this estimate, one can reduce the problem of counting the Dirichlet eigenvalues in the region $j_{n,k}>(1+c_0)|n|$ into the problem of counting corresponding lattice points (translated by $(0,-1/4)$), see (\ref{part1}) below. 
It remains to count eigenvalues inside the region $j_{n,k} <(1+C_0)|n|$ (with $C_0>c_0$, which will be specified below). In view of Lemma \ref{F-nk}, it is enough to work on points with $n>0$. F. Olver (\cite{olver}) obtained the following asymptotic formula
\begin{equation}\label{jnknearboundary}
j_{n,k} = n(1+\psi(t_k/n^{2/3}))+O(n^{-1}),
\end{equation}
where $\psi$ is a smooth function with $\psi(0)=0, \psi'(0)>0$, and $t_k$ is the $k$-th negative zero of the Airy function. 

For zeros of Bessel functions, we also have the following well-known upper and lower bounds estimates (see, for example, (5.3) in \cite{CGJ})
\[A_1 \left(k+k^{2/3}n^{1/3} \right) \le j_{n,k}-n \le A_2 \left(k+k^{2/3}n^{1/3}\right),\]
where $A_1, A_2$ are positive constants. We will let $\tilde c_0, \widetilde C_0$ be the solutions to the following equations
\[A_1(\tilde c_0+\tilde c_0^{2/3})=c_0 \quad \text{and} \quad A_2(\widetilde C_0+\widetilde C_0^{2/3}) = C_0.\]
It is easy to check 
\begin{enumerate}
\item If $j_{n,k}<(1+c_0)n$, then $k <\tilde c_0 n $.
\item If $k < \widetilde C_0  n$, then $j_{n,k}<(1+  C_0) n $. 
\end{enumerate}
Then in what follows, we will choose $C_0$ small enough so that Lemma \ref{Nkerror} holds, and choose $c_0$ small so that $\tilde c_0<\widetilde C_0$. 

For simplicity we only consider the parts with $n>0$. 
Using the asymptotic formula (\ref{jnknearboundary}), Colin de Verdi\`{e}re was able to prove
\begin{lemma}[\cite{colin}]\label{Nkerror} For each $k \ge 1$, let
\[\aligned
N_k(\mu) = \#A_{k,\mu}&:=\#\{n\in\mathbb{Z}:(n,k- {1}/{4})\in\mu\Omega,\, k\le \widetilde C_0n\}, \\
\mathcal N_k(\mu) = \#\mathcal A_{k,\mu}&:=\#\{n\in\mathbb{Z}:j_{n,k}\le\mu,\,   k\le \widetilde C_0 n\},
\endaligned\]
then for $C_0$ small enough there exists constant $C$ so that for all $k$,
\[|N_k(\mu)-\mathcal{N}_k(\mu)|\le N_k(\mu+ {C}/{\mu})-N_k(\mu-{C}/{\mu})+C\mu^{1/3}k^{-4/3}.\]
\end{lemma}

Thanks to these lemmas and Theorem \ref{lattice-thm}, we are able to prove the following proposition that gives us a nicer bound on the error between $N_{\Omega}(\mu)$ and $\mathcal{N}_{disk}(\mu)$ than the one in \cite{colin}.  
\begin{proposition}\label{error} There exists a constant $C$ such that
\[|\mathcal{N}_{disk}(\mu)-N_\Omega(\mu)|\le 3(N_\Omega(\mu+{C}/{\mu})-N_\Omega(\mu- {C}/{\mu}))+O(\mu^{1/3}).\]
\end{proposition}
\begin{proof}
As in \cite{colin}, we fix a smooth cut-off function $\chi$ which is even and compactly supported in $(-\widetilde C_0, \widetilde C_0)$, such that $\chi \equiv 1$ in the interval $\left[-\tilde c_0, \tilde c_0\right] \subset (-1, 1)$. Using $\chi$ one splits $N_\Omega(\mu)$ into three parts, according to whether a point is close to the two boundary rays, and split  $\mathcal N_{disk}(\mu)$ into three parts in the same fashion,
\[\aligned
N_\Omega(\mu) &=N^1_\Omega(\mu) +N^2_\Omega(\mu) +N^3_\Omega(\mu),\\
\mathcal N_{disk}(\mu) & = \mathcal N^1_{disk}(\mu)  + \mathcal N^2_{disk}(\mu) +\mathcal N^3_{disk}(\mu),
\endaligned\]
where \footnote{The $\chi$'s in the splittings that appeared in \cite[page 8]{colin} should be $1-\chi$.}
\[\aligned
&N^1_{\Omega}(\mu)=\sum_{(n,k+\max(0,-n)-1/4)\in\mu \Omega}(1-\chi(k/n)),
\\ &N^2_{\Omega}(\mu)=\sum_{n\ge 0,(n,k-1/4)\in\mu \Omega}\chi(k/n),
\\&  N^3_{\Omega}(\mu)=\sum_{n< 0,(n,k+|n|-1/4)\in\mu \Omega}\chi(k/n)\endaligned\]
and
\[\aligned
& \mathcal N^1_{disk}(\mu) =\sum_{j_{n,k}\le\mu}(1-\chi(k/n)),
\\ & \mathcal N^2_{disk}(\mu)=\sum_{n\ge 0,j_{n,k}\le \mu}\chi(k/n),
\\& \mathcal N^3_{disk}(\mu)=\sum_{n< 0,j_{n,k}\le \mu}\chi(k/n).
\endaligned\]
In the above expression, we will take $\chi(k/n)=0$ for $n=0$. 
Note that by the evenness of $\chi$, one has $\mathcal N^2_{disk}(\mu)=\mathcal N^3_{disk}(\mu)$. Moreover, in view of Lemma \ref{F-nk}, $N^2_\Omega(\mu)=N^3_\Omega(\mu)$. 

For the ``inner part", {\it i.e.} the parts with superscript 1, one has $1-\chi(k/n)=0$ for $j_{n,k}<(1+c_0)|n|$. By using Lemma \ref{sta} it was shown in \cite{colin} that
\begin{equation}\label{part1}
N_\Omega^1(\mu- {C}/{\mu})\le \mathcal{N}_{disk}^1(\mu)\le
N_\Omega^1(\mu+ {C}/{\mu}).
\end{equation}
In particular, we get
\[
|\mathcal N_{disk}^1(\mu)-N_\Omega^1(\mu)| \le N_\Omega^1(\mu+C/\mu) - N_\Omega^1(\mu-C/\mu).
\]
For the parts with superscript 2, in view of the fact $\chi(  k/n)=0$ for $k>\widetilde C_0n$, one has
\[\aligned
|N_\Omega^2(\mu) -\mathcal N_{disk}^2(\mu)| & \le
\sum_{k=1}^\infty
\left|
 \sum_{n \in A_k(\mu)}\chi(  k/n) - \sum_{n \in \mathcal A_k(\mu)}\chi( k/n)
\right|
\\& \le \sum_k |  N_k(\mu) -  \mathcal N_k(\mu)|,
\endaligned\]
where in the last step we used the fact that for each $k$, either $A_{k,\mu} \subset \mathcal A_{k,\mu}$, or $\mathcal A_{k,\mu} \subset A_{k,\mu}$.
So  if we denote
\[A_\mu  := \{(k,n)\ |\ k\ge 1, n \in A_{k,\mu}\},\]
then by definition, we have
\[\aligned
& \sum_k\left(N_k(\mu+{C}/{\mu})-N_k(\mu-{C}/{\mu})\right) \\
= &  N_\Omega^2(\mu+ C/\mu)-N_\Omega^2(\mu- C/\mu)+ \sum_{(k,n) \in A_{\mu+ C/\mu} \setminus A_{\mu- C/\mu}}(1 -  \chi( k/n))
\\ \le & N_\Omega^2(\mu+ C/\mu)-N_\Omega^2(\mu- C/\mu)+ \# A_{\mu+ C/\mu}\setminus A_{\mu- C/\mu}
\\
\le &  N_\Omega^2(\mu+ C/\mu)-N_\Omega^2(\mu- C/\mu) +  N_\Omega(\mu+C/\mu) - N_{\Omega}(\mu-C/\mu).
\endaligned\]
Combining the above two inequalities with  Lemma \ref{Nkerror}, we get
\[\aligned
|\mathcal{N}_{disk}^2(\mu)-N_\Omega^2(\mu)|   \le & N_\Omega^2(\mu+ C/\mu)-N_\Omega^2(\mu- C/\mu)   +N_\Omega(\mu+C/\mu)  \\
&- N_{\Omega}(\mu-C/\mu)+O(\mu^{1/3}).
\endaligned\]
The same bound apply to the parts with superscripts 3, and we thus get
\[
|\mathcal{N}_{disk}(\mu)-N_\Omega(\mu)|  \le  3( N_\Omega(\mu+ {C}/{\mu})-N_\Omega(\mu-{C}/{\mu}))+O(\mu^{1/3}).
\]
\end{proof}

Now we can prove
\begin{proof}[Proof of Theorem \ref{spectrum-thm}]
According to (\ref{remainder}) and Theorem \ref{lattice-thm}, one  has
\[
N_\Omega(\mu) = \pi \mu^2 -\mu/2+O(\mu^{2/3-1/495}).
\]
Plug this into Proposition \ref{error}, we immediately get
\[\aligned
|\mathcal R_{disk}(\mu)| & = |\mathcal N_{disk}(\mu) -\pi \mu^2+\mu/2| \\
& = 3( N_\Omega(\mu+ {C}/{\mu})-N_\Omega(\mu-{C}/{\mu}))+ O(\mu^{2/3-1/495})
\\
& = O(\mu^{2/3-1/495}).
\endaligned\]
\end{proof}

\begin{remark}\label{neumann}
Similarly, for the Neumann boundary condition case, the eigenvalues are given  by the zeroes $j'_{n,k}$'s of the derivative of the Bessel functions $J_n$'s, and one can prove
\[ j^{\prime}_{n,k}=\left\{
\begin{array}{ll}
F(n,k-\frac{3}{4})+O(\frac{1}{n+k+1}),\;&\mathrm{if}\;n\ge 0,\\
F(n,k+|n|-\frac{3}{4})+O(\frac{1}{|n|+k+1}),\;&\mathrm{if}\;n<0
\end{array}
\right.\]
for $j^{\prime}_{n,k}>(1+\tilde{c})|n|$. As a consequence, one can convert the counting problem for $\mathcal N_{disk}(\mu)$ with Neumann boundary condition to the counting problem for $N_\Omega(\mu)$ with $(a, b)=(0, -3/4)$.
Essentially the same argument as above gives
$\mathcal R^{Neumann}_{disk}(\mu) = O(\mu^{  2/3-  1/{495}})$. \end{remark}

%


\appendix
\section{A Useful Lemma}\label{app1}

Here is a quantitative version of the inverse function theorem. It
is routine to prove it by following a standard proof of the theorem.

\begin{lemma}\label{app:lemma:1}
Suppose that $f$ is a $C^{(k)}$ ($k\geqslant 2$) mapping from an
open set $\Omega\subset\mathbb{R}^d$ into $\mathbb{R}^d$ and
$b=f(a)$ for some $a\in \Omega$. Assume $|\det (\nabla f(a))|$
$\geqslant$ $c$ and for any $x\in\Omega$,
\begin{equation*}
|D^{\nu} f_i(x)|\leqslant C \quad \quad \textrm{for $|\nu|\leqslant
2$, $1\leqslant i\leqslant d$}.
\end{equation*}
If $r_0\leqslant \sup\{r>0: B(a, r)\subset \Omega\}$ then $f$ is
bijective from $B(a, r_1)$ to an open set containing $B(b, r_2)$
where
\begin{equation*}
r_1=\min\left\{\frac{c}{2d^{2} d! C^d}, r_0\right\}  \quad \textrm{and}\quad
r_2=\frac{c}{4d!C^{d-1}}r_1.
\end{equation*}
The inverse mapping $f^{-1}$ is also in $C^{(k)}$.
\end{lemma}



\end{document}